\documentclass[11pt]{amsart}
\usepackage[left=1in, right=1in, top=1.2in, bottom=1.2in]{geometry}
\usepackage[utf8]{inputenc}
 \usepackage{graphicx}
 \usepackage{amssymb}
 \usepackage{comment}
 \usepackage[mathscr]{euscript}
 \usepackage{tikz}
\usetikzlibrary{decorations.pathmorphing, patterns,shapes,arrows,automata,positioning}
\usepackage{amsmath}
\usepackage{amsfonts}
\usepackage{mathtools}
\usepackage{amsthm}
\usepackage{hyperref}
\usepackage{enumerate}
\usepackage{indentfirst}
\usepackage{tikz}
\usetikzlibrary{arrows,automata,positioning}
\usepackage{color}

\newtheorem{thm}{Theorem}

\newtheorem{lem}[thm]{Lemma}
\newtheorem{fact}[thm]{Fact}
\newtheorem{cor}[thm]{Corollary}

\newtheorem*{thm*}{Theorem}
\newtheorem{innercustomthm}{Theorem}
\newenvironment{customthm}[1]
  {\renewcommand\theinnercustomthm{#1}\innercustomthm}
  {\endinnercustomthm}

\numberwithin{thm}{section}

\newtheorem{conjecture}[thm]{Conjecture}

\theoremstyle{definition}
\newtheorem{defn}[thm]{Definition}

\newtheorem{rmk}[thm]{Remark}
\newcommand{\N}{\mathbb{N}}
\newcommand{\R}{\mathbb{R}}

\newcommand{\Z}{\mathbb{Z}}

\newcommand{\calI}{\mathcal{I}}
\newcommand{\calE}{\mathcal{E}}
\newcommand{\calX}{\mathcal{X}}
\newcommand{\calN}{\mathcal{N}}

\title{Sublinear Michaux-Villemaire over $(\N, +, k^\N)$}

\author{Chris Schulz}
\address{Max Planck Institute for Software Systems \\
Saarland, Germany}
\email{chschulz@mpi-sws.org}


\begin{document}

\begin{abstract}
    In 1996, Michaux and Villemaire showed that any nontrivial expansion of $(\N, +)$ defines new unary sets. We conjecture that the same holds over the structure $(\N, +, k^\N)$, and we prove a partial result in this direction. To do so, we first show that every nonempty sparse $k$-automatic set is a Presburger-definable image of some Cartesian power of $k^\N$. Using this characterization, we prove that if an expansion of $(\N, +, k^\N)$ defines the same unary sets, then every set of sublinear growth definable in the expansion is already definable in $(\N, +, k^\N)$.
\end{abstract}

\maketitle

\section{Introduction}

The question of definability in structures on the natural numbers has received considerable attention in the past several decades. Let $\calN = (\N, \dots)$ be some first-order structure on the natural numbers. It is natural to ask precisely which relations $X \subseteq \N^d$ are definable in $\calN$, i.e. which relations may be defined by a first-order formula in $\calN$. In this paper, as in many others, we focus specifically on $\calN$ of the form $(\N, +, \dots)$, i.e. expansions of \textbf{Presburger arithmetic}. In these structures, the distinction between $\varnothing$-definability and definability with parameters vanishes, because every constant is $\varnothing$-definable in $\calN$.

Presburger arithmetic itself is well-understood: the sets it defines are precisely the semilinear sets (\cite{GS66S}). The following are some examples of well-studied expansions of Presburger arithmetic and their properties:

\begin{itemize}
    \item \textbf{True arithmetic} is the structure $(\N, +, \cdot)$. The definable sets in this structure are called \textbf{arithmetical}, and they include most sets of natural numbers studied in mathematics.
    \item \textbf{$k$-B\"uchi arithmetic} is the structure $(\N, +, V_k)$, where $(x, y) \in V_k$ when $y$ is the largest power of $k$ dividing $x$. It was shown by Bruy\`ere et al. in \cite{BHMV94} that the sets definable in $k$-B\"uchi arithmetic are precisely the \textbf{$k$-automatic} sets, i.e. the sets encoded by some regular language. Another famous result on this structure is the \textbf{Cobham-Semenov theorem}, which states that if $k_1$ and $k_2$ have no nontrivial common powers, then the only sets defined in both $k_1$- and $k_2$-B\"uchi arithmetic are the Presburger-definable sets. This was proven by Cobham in the one-dimensional case (\cite{C69}) and by Semenov in the general case (\cite{Sem77}).
    \item \textbf{$k$-Semenov arithmetic} is the structure $(\N, +, n \mapsto k^n)$. By Semenov (\cite{Sem84}), this structure has a decidable first-order theory.
    \item A common reduct of both $k$-B\"uchi and $k$-Semenov arithmetic is $(\N, +, k^\N)$, where $k^\N$ is the set $\{k^n : n \in \N\}$. B\`es showed in \cite{Bes97} that a set definable in this structure is definable in \textit{any} nontrivial expansion $(\N, +, E)$ where $E$ is $k$-automatic.
    \item We may combine the above predicates in multiple bases. If $k_1$ and $k_2$ have no nontrivial common powers, then the structure $(\N, +, V_{k_1}, k_2^\N)$ defines multiplication and is thus equivalent to true arithmetic in terms of first-order definability (\cite{Bes97}). The structure $(\N, +, k_1^\N, k_2^\N)$ has an undecidable first-order theory (\cite{HS22}), but it does not define multiplication, and indeed no $(\N, +, k_1^\N, \dots, k_r^\N)$ defines multiplication (\cite{Sch25}).
\end{itemize}

In \cite{MV96}, Michaux and Villemaire proved the following result on expansions of Presburger arithmetic.

\begin{fact}[\cite{MV96}, Theorem 5.1]
    Let $E \subseteq \N^d$, and assume that $(\N, +, E)$ and $(\N, +)$ define the same subsets of $\N$. Then $E$ is Presburger-definable.
\end{fact}

It is natural to ask what other structures have the same property. So we may say that any structure $\calX = (X, \dots)$ has the \textbf{Michaux-Villemaire property} if for all $E \subseteq X^d$, whenever $(\calX, E)$ defines the same subsets of $X$ as $\calX$ itself, the set $E$ is already definable in $\calX$. Michaux and Villemaire's result thus states that Presburger arithmetic has the Michaux-Villemaire property.

Other structures satisfying this property include $(\N, +, \cdot)$ and in fact any $\calX$ with a definable injection $X^2 \hookrightarrow X$. If we have such an injection, then finite tuples may be encoded by individual elements of $X$, and therefore every higher-arity set is interdefinable with a unary set. On the other hand, consider the structure $(\R, <, +, \cdot, \exp)$. By \cite{Wil96}, this structure is \textbf{o-minimal}, meaning that every subset of $\R$ definable in this structure is definable in $(\R, <)$. Therefore, the structure $(\R, <)$ and intermediate structures such as $(\R, <, +)$ and $(\R, <, +, \cdot)$ fail the Michaux-Villemaire property.

The motivation behind this paper is the following conjecture.

\begin{conjecture}
\label{sparse_definable_mv}
    The structure $(\N, +, k^\N)$ has the Michaux-Villemaire property.
\end{conjecture}

If true, this conjecture would have immediate implications. For instance, Bell, Block Gorman, and Schulz recently proved a dichotomy result for unary $k$-automatic sets.

\begin{fact}[\cite{BBS26}, Theorem 1.1]
\label{bbs_dichotomy}
    Let $X \subseteq \N$ be a $k$-automatic set. Then either $X$ is definable in $(\N, +, k^\N)$, or else $(\N, +, X)$ defines $V_k$ (and hence every $k$-automatic set).
\end{fact}

If $(\N, +, k^\N)$ has the Michaux-Villemaire property, then we immediately obtain a generalization of the above to $X \subseteq \N^d$.

In this paper, we do not provide a proof or refutation of the above conjecture. Instead, we obtain the following partial result.

\begin{customthm}{\ref{sublinear_e_definable}}
    Let $k > 1$, and let $\calE$ be an expansion of $(\N, +, k^\N)$ that defines the same unary sets. Let $X \subseteq \N^d$ be definable in $\calE$, and assume that $|X \cap [0, n]^d| \in o(n)$. Then $X$ is definable in $(\N, +, k^\N)$.
\end{customthm}

In particular, letting $\calE = (\N, +, k^\N, X)$, we obtain the Michaux-Villemaire property for $(\N, +, k^\N)$ restricted to expansions by sets with sublinear growth.

While Conjecture \ref{sparse_definable_mv} is suspected to hold, it is also plausible that the sublinear bound given by Theorem \ref{sublinear_e_definable} is tight. The reason relates to Corollary \ref{entropy_dichotomy}: every sparse-definable set is either sparse or has $\Theta(n)$ growth. However, in the higher-dimensional case $X \subseteq \N^d$, such a dichotomy fails; indeed, such a set can have $\Theta(n^r)$ growth for any $0 \leq r \leq d$. (For example, the graph of addition is a sparse-definable set with $\Theta(n^2)$ growth.) So it is plausible that there exists some higher-dimensional set $X$ with $\Theta(n)$ growth (for instance, the graph of a linearly bounded function $f : \N \to \N$) that is intermediate between $k^\N$ and $V_k$ in terms of the sets it defines, thus disproving the higher-dimensional generalization of Fact \ref{bbs_dichotomy}.

Along the way toward proving Theorem \ref{sublinear_e_definable}, we obtain a characterization theorem for sparse $k$-automatic sets (here \textit{sparse} means that a set has at most polylogarithmic growth).

\begin{customthm}{\ref{sparse_equals_power_internal}}
    Let $X \subseteq \N^d$. The following are equivalent:

    \begin{enumerate}[\normalfont (I)]
        \item $X$ is a sparse $k$-automatic set.
        \item There exist $r > 0$, a Presburger-definable $D \subseteq \N^r$, and a Presburger-definable $f : D \to \N^d$ such that $f(D \cap (k^\N)^r) = X$.
        \item There exist $r > 0$, a $k$-automatic $D' \subseteq (k^\N)^r$, and a $k$-automatic $f : D' \to \N^d$ such that $f(D') = X$.
    \end{enumerate}
\end{customthm}

\textbf{Acknowledgements.} The author was supported in part by ERC Grant DynAMiCs (ID. 101167561). The author would like to thank Jason Bell, Alexi Block Gorman, and Mikhail Starchak for helpful conversations about this paper.

While composing this paper, the author made use of LLM assistance for the purposes of literature review, proof generation, and proofreading, using GPT-6. The text of this paper was typed by hand; none of its verbiage was LLM-generated.

\section{Preliminaries}

Fix an integer $k > 1$. Throughout, definability is with parameters from the entire domain (though as mentioned before, we work in expansions of Presburger arithmetic, so it will not matter). The set $\N$ contains $0$.

The study of sets definable from base-$k$ arithmetic naturally makes use of the base-$k$ encoding itself. In this paper, we will define the base-$k$ encoding as follows.

\begin{defn}
    Let $A$ be any set.
    \begin{enumerate}[(i)]
        \item For a tuple $\vec{x} = (x_1, \dots, x_d) \in A^d$, let $\pi_i(\vec{x}) = x_i \in A$.
        \item For a subset $X \subseteq A^d$, let $\pi_i(X) = \{\pi_i(\vec{x}) : \vec{x} \in X\} \subseteq A$.
        \item Lastly, for a word $w = \vec{x}_1 \dots \vec{x}_n \in (A^d)^*$, let $\pi_i(w) = \pi_i(\vec{x}_1) \dots \pi_i(\vec{x}_n) \in A^*$.
    \end{enumerate}
    That is, $\pi_i$ denotes \textbf{projection} onto the $i$th coordinate.
\end{defn}

\begin{defn}
    \,
    \begin{enumerate}[(i)]
        \item We denote by $\Sigma_k = \{0, \dots, k-1\}$ the alphabet of digits in base $k$.
        \item Given a word $w = w_n w_{n-1} \dots w_1 w_0 \in \Sigma_k^*$ (note the order!), we let:

        $$[w]_k = \sum_{i=0}^n w_i k^i$$

        In other words, $[w]_k$ is the number that has $w$ as a base-$k$ representation, with the \textbf{most significant digit first}. We say that $w$ \textbf{encodes} $[w]_k$. The \textbf{canonical} encoding of a natural number is its shortest such representation (so, for zero, it is the empty word).
        \item Let $d$ be a positive integer. Given a word $w \in (\Sigma_k^d)^*$, we let:

        $$[w]_k = ([\pi_1(w)]_k, \dots, [\pi_d(w)]_k)$$

        In other words, $[w]_k$ is the tuple in $\N^d$ whose $i$th coordinate is encoded by the $i$th coordinate of each character of $w$.
        \item Let $d$ be a positive integer. Given a language $L \subseteq (\Sigma_k^d)^*$, we let $[L]_k = \{[w]_k : w \in L\}$.
    \end{enumerate}
\end{defn}

It follows elementarily that every tuple in $\N^d$ has exactly one canonical encoding in $(\Sigma_k^d)^*$, and all other encodings are simply the canonical encoding prefixed by some number of $\vec{0} = (0, \dots, 0)$. So when we find an $L$ encoding a given set, we will normally require that $L$ be \textbf{zero-closed}, meaning that it is closed under the operation of prefixing $\vec{0}$.

For more on base-$k$ encodings of sets, the reader may consult \cite{AS03}.

We now define several notions of definability.

\begin{defn}
    \,
    \begin{enumerate}[(i)]
        \item A set $X \subseteq \N^d$ is \textbf{$k$-automatic} if the language of all encodings of its tuples is regular. Equivalently (\cite{BHMV94}), $X$ is $k$-automatic if it is definable in $(\N, +, V_k)$, where $(x, y) \in V_k$ when $y$ is the largest power of $k$ dividing $x$.
        \item A set $X \subseteq \N^d$ is \textbf{sparse-definable} if it is definable in $(\N, +, k^\N)$.
        \item A set $X \subseteq \N^d$ is \textbf{Presburger-definable} if it is definable in $(\N, +)$.
    \end{enumerate}

    We say that a function $f : D \to \N^d$, where $D \subseteq \N^r$, is $k$-automatic (resp. sparse-definable, Presburger-definable) if its graph is $k$-automatic (resp. sparse-definable, Presburger-definable).
\end{defn}

Clearly, (iii) is stronger than both (i) and (ii); in fact, it is strictly stronger, because $k^\N$ is not Presburger-definable. Moreover, (ii) is strictly stronger than (i), as may be seen via a number of methods. Perhaps the most direct is to note that in \cite{Sem80}, Semenov gives a characterization theorem for sparse-definable subsets of $\N$, and that there exist $k$-automatic sets failing this characterization.

Much of this paper will involve measuring the sizes of sets and languages. We do this via what we call \textbf{growth functions}.

\begin{defn}
    \,
    \begin{enumerate}[(i)]
        \item Let $X \subseteq \N^d$ be a set. The \textbf{growth function} of $X$ maps a natural number $n$ to $|X_{\leq n}|$, where $X_{\leq n} = \{(x_1, \dots, x_d) \in X : \max(x_1, \dots, x_d) \leq n\}$. We say that $X$ is \textbf{sparse} if $|X_{\leq n}| \in O(\log^r n)$ for some natural number $r$.
        \item Let $L \subseteq \Sigma^*$ be a language. The \textbf{growth function} of $L$ maps a natural number $n$ to $|L_{\leq n}|$, where $L_{\leq n} = \{w \in L : |w| \leq n\}$. We say that $L$ is \textbf{sparse} if $|L_{\leq n}| \in O(n^r)$ for some natural number $r$.
    \end{enumerate}
\end{defn}

\begin{rmk}
\label{sparse_iff_sparse}
    Let $L \subseteq (\Sigma_k^d)^*$ be zero-closed, and let $X = [L]_k$. Then:
    
    $$|L_{\leq n}| = \sum_{i=0}^n |X_{\leq k^i - 1}|,$$

    and the latter expression is bounded below by $|X_{\leq k^n - 1}|$ and above by $(n+1) |X_{\leq k^n - 1}|$. From this it follows that $L$ is sparse iff $X$ is sparse.
\end{rmk}

Szilard, Sheng, Zhang, and Shallit give a decomposition theorem for sparse regular languages in \cite{SSZS92}.

\begin{fact}[\cite{SSZS92}, Theorem 3]
\label{sparse_char}
    Let $L$ be a regular language. Then $L$ is sparse iff it may be written as a finite union of languages of the form $v_1w_1^* v_2w_2^* \cdots v_rw_r^* v_{r+1}$, where each $v_i$ is a word and each $w_i$ is a nonempty word.
\end{fact}

We now motivate our use of ``sparse-definable'' to mean definable in $(\N, +, k^\N)$. In fact, sparseness and definability in $(\N, +, k^\N)$ are closely related concepts. The relationship follows from a characterization of unary sparse-definable sets given by Semenov.

\begin{fact}[\cite{Sem80}, Theorem 5]
\label{semenov_char}
    Let $X \subseteq \N$ be sparse-definable. Then $X$ is encoded by a finite union of languages of the form $v_1w_1^* v_2w_2^* \cdots v_rw_r^* \Sigma_{\ell, m, c}$, where $v_i, w_i$ are words and where $\Sigma_{\ell, m, c}$, for natural numbers $\ell, m, c$, is the set of words of length divisible by $\ell$ that encode numbers congruent to $-c$ modulo $m$.
\end{fact}

\begin{cor}
\label{entropy_dichotomy}
    \,
    \begin{enumerate}[(i)]
        \item Let $X \subseteq \N$ be sparse-definable. Then either $X$ is sparse, or else $|X_{\leq n}| \in \Theta(n)$.
        \item Let $X \subseteq \N$ be sparse $k$-automatic. Then $X$ is sparse-definable.
    \end{enumerate}
\end{cor}
\begin{proof}
    For (i), it suffices to prove the result for a set $X$ encoded by an individual language of the form $L = v_1w_1^* v_2w_2^* \cdots v_rw_r^* \Sigma_{\ell, m, c}$. If $\Sigma_{\ell, m, c}$ is finite, then $X$ is sparse by Fact \ref{sparse_char}. Otherwise, let $L' = v_1 \dots v_r \Sigma_{\ell, m, c} \subseteq L$, and consider for $n \geq 0$ the set of words in $L'$ of length $|v_1 \dots v_r| + n\ell$. The number of words in this set is equal to the number of words of length $n\ell$ that encode numbers congruent to $-c$ modulo $m$, which is at least $\left\lfloor \frac{k^{n\ell}}{m} \right\rfloor$. Each such word encodes a distinct element of $X$ less than $k^{n\ell}$, so:

    $$|X_{\leq k^{nl}}| \geq \left\lfloor \frac{k^{n\ell}}{m} \right\rfloor \in \Omega(k^{nl}).$$

    It remains to note that $X$ cannot possibly have a growth function outside $O(n)$ because it is bounded above by the growth function of $\N$ itself.

    For (ii), let $L$ be the language encoding $X$. The result follows by applying Facts \ref{sparse_char} and \ref{semenov_char}.
\end{proof}

Therefore, any subset of $\N^d$ definable from Presburger and \textit{unary} sparse $k$-automatic sets is in fact sparse-definable by our definition. (Corollary \ref{sparse_sparse_definable} generalizes this to higher-dimensional $X$.)

Sparseness is convenient because it provides a natural set-theoretic ideal to use in the context of $k$-automatic sets. We will make use of this ideal in conjunction with the following previous result of Schulz \cite{Sch25}. (Note that the original statement of the following is stronger, and moreover is over $\Z$ instead of $\N$, but the author notes that this does not matter.)

\begin{fact}[\cite{Sch25}, Theorem C]
\label{schulz_definable_dichotomy}
    Let $\calI$ be any set-theoretic ideal on $\N$. Let $E \subseteq \N$ be such that any Presburger-definable image of a power of $E$ lies in $\calI$. Then every unary set definable in $(\N, +, E)$ either lies in $\calI$ or else contains arbitrarily long arithmetic progressions with the same difference.
\end{fact}

\section{Internally witnessed sparseness}

The goal of this section will be to prove Theorem \ref{sparse_equals_power_internal}. We begin by noting that $k$-automatic functions cannot grow too quickly (or too slowly!) and hence preserve sparse sets. A version of this for languages themselves was originally proven by Elgot and Mezei in \cite{EM65}; see also Proposition 6.1 of Blumensath and Gr\"adel (\cite{BG04}). For completeness, we provide a direct proof here.

\begin{fact}
\label{kautomatic_linear_growth}
    Let $f : \N^r \to \N^d$ be a $k$-automatic function. Then there exists some $C$ such that whenever $f(x_1, \dots, x_r) = (y_1, \dots, y_d)$, we have $\max(y_1, \dots, y_d) \leq C \max(1, x_1, \dots, x_r)$.
\end{fact}
\begin{proof}
    Let $L$ be the regular language with alphabet $\Sigma_k^{r+d}$ encoding the graph of $f$. Apply the pumping lemma to $L$, and let $p$ be its pumping length. Let $C = k^{p+1}$.

    Assume the lemma is false. In particular, it must then be false for some nonzero tuple $(x_1, \dots, x_r)$, as otherwise we could  simply increase the value of $C$. So let $f(x_1, \dots, x_r) = (y_1, \dots, y_d)$ be such that $\max(y_1, \dots, y_d) > C \max(1, x_1, \dots, x_r) > 0$. In particular, let $y_i$ and $x_j$ be the maximal coordinates; then $y_i > k^{p+1} x_j > 0$. Then $L$ contains a word $w$ of tuples whose $(r+i)$th coordinates encode $y_i$ and whose length is equal to the number of digits in $y_i$. As $x_j$ is less than $\frac{y_i}{k^{p+1}}$, the first $r$ coordinates of $w$ are zero for the first $p$ symbols.

    By the pumping lemma, we can write $w = xyz$, where $|xy| \leq p$ and $|y| \geq 1$, such that $xz \in L$. Now, $w$ is a word over $(\Sigma_k^{r+d})^*$; write $w_L, w_R$ to mean the words containing the first $r$ coordinates and last $d$ coordinates of each tuple in $w$, respectively, and write $x_L, y_L, z_L, x_R, y_R, z_R$ correspondingly.

    Note that because $|xy| \leq p$, $x_L, y_L$ contain only zeroes. Therefore, $[x_L z_L]_k = [x_L y_L z_L]_k$. However, the first character of $w_R$ is nonzero. Therefore, $xz \in L$ encodes $(x_1, \dots, x_r, z_1, \dots, z_d)$ where $z_i \neq y_i$. But this means that the graph of $f$ is not a functional relation; contradiction.
\end{proof}

\begin{cor}
\label{kautomatic_preserves_sparse}
    Let $f : \N^r \to \N^d$ be a $k$-automatic function. If $D \subseteq \N^r$ is sparse and $k$-automatic, then $f(D)$ is sparse.
\end{cor}
\begin{proof}
    For each $\vec{x} \in f(D)$, let $g(\vec{x})$ be the lexicographically first preimage of $\vec{x}$ in $D$. The function $g$ is $k$-automatic, so by Fact \ref{kautomatic_linear_growth}, there is $C \in \N$ such that the coordinates of $g(\vec{x})$ are at most $C$ times the maximum of $1$ and the greatest coordinate of $\vec{x}$. Therefore, for $N$ sufficiently large, the number of points in $f(D)$ with maximum coordinate at most $N$ is bounded by the number of points in $D$ with maximum coordinate at most $CN$. This is in $O(\log^r CN) = O(\log^r N)$, so it follows that $f(D)$ is sparse.
\end{proof}

This allows us to consider the notion of sparseness for a $k$-automatic set that is \textit{internally witnessed} by a surjection from some power of $k^\N$. In other words, the most fundamental example of a sparse $k$-automatic set is $(k^\N)^r$ for some $r$, and by the above corollary, any $k$-automatic transformation applied to this set preserves its sparseness. So it is natural to ask whether \textit{every} other sparse $k$-automatic set is a transformed version of $(k^\N)^r$, for some sufficient set of transformations. We answer this question in the affirmative.

\begin{thm}
\label{sparse_equals_power_internal}
    Let $X \subseteq \N^d$. The following are equivalent:

    \begin{enumerate}[\normalfont (I)]
        \item $X$ is a sparse $k$-automatic set.
        \item There exist $r > 0$, a Presburger-definable $D \subseteq \N^r$, and a Presburger-definable $f : D \to \N^d$ such that $f(D \cap (k^\N)^r) = X$.
        \item There exist $r > 0$, a $k$-automatic $D' \subseteq (k^\N)^r$, and a $k$-automatic $f : D' \to \N^d$ such that $f(D') = X$.
    \end{enumerate}
\end{thm}
\begin{proof}
    The implication (II) $\implies$ (III) is trivial, and (III) $\implies$ (I) follows immediately from Corollary \ref{kautomatic_preserves_sparse}, so we need only show (I) $\implies$ (II). Let $X$ be a sparse $k$-automatic set; by Remark \ref{sparse_iff_sparse}, $X$ is encoded by a zero-closed sparse regular $L \subseteq (\Sigma_k^d)^*$. Then by Fact \ref{sparse_char}, we may write $L = L_1 \cup L_2 \cup \dots \cup L_m$, where each $L_i$ is of the form:

    $$L_i = u_1 v_1^* u_2 v_2^* \dots u_r v_r^* u_{r+1}$$

    Note that $L$ is zero-closed, so it makes no difference to assume each $L_i$ is zero-closed (by replacing $L_i$ with $\vec{0}^* L_i$) and hence to assume $r > 0$. We will first give $D_i \subseteq \N^r$ and $f_i : D_i \to \N^d$ such that $f_i(D_i \cap (k^\N)^r) = [L_i]_k$. For a word $w \in L_i$, i.e. $w = u_1 v_1^{e_1} u_2 v_2^{e_2} \dots u_r v_r^{e_r} u_{r+1}$, and for $1 \leq j \leq r$, let $\ell_j = |v_j^{e_j} u_{j+1} \dots u_r v_r^{e_r} u_{r+1}|$. We note that the finite sequence $(\ell_j)_j$ uniquely determines $w$, and its admissible values are characterized thus: letting $\ell_{r+1} = 0$, we must have that $|v_j^{e_j}| = \ell_j - |u_{j+1}| - \ell_{j+1}$ is nonnegative and a multiple of $|v_j|$. As long as these conditions are followed, we obtain a one-to-one correspondence between the vectors $(\ell_j)_j$ and the words $w$.

    So define $D'_i$ to be the set of all tuples $(k^{\ell_1}, \dots, k^{\ell_r}) \in (k^\N)^r$ such that, again setting $\ell_{r+1} = 0$, we have for each $j$ that $k^{\ell_j}$ is not less than $k^{|u_{j+1}|} k^{\ell_{j+1}}$ and that the difference in their exponents is a multiple of $|v_j|$. The latter condition is equivalent to requiring that $k^{\ell_j}$ and $k^{|u_{j+1}|} k^{\ell_{j+1}}$ are congruent modulo $k^{|v_j|} - 1$. We see that $D'_i$ is the intersection of $(k^\N)^r$ with a Presburger-definable set $D_i$.

    Now let $1 \leq a \leq d$ and notice that:
    
    $$[\pi_a(u_j) 0^{\ell_j}]_k = [\pi_a(u_j)]_k \cdot k^{\ell_j}$$

    and:
    
    \begin{equation*}
        \begin{split}
            [\pi_a(v_j)^{e_j} 0^{|u_{j+1}| + \ell_{j+1}}]_k & = \frac{[\pi_a(v_j)]_k}{[1^{|v_j|}]_k} \cdot [(1^{|v_j|})^{e_j} 0^{|u_{j+1}| + \ell_{j+1}}]_k \\
            & = \frac{[\pi_a(v_j)]_k}{[1^{|v_j|}]_k} \cdot \left(k^{|v_j| e_j + |u_{j+1}| + \ell_{j+1}} - k^{|u_{j+1}| + \ell_{j+1}}\right) \\
            & = \frac{[\pi_a(v_j)]_k}{[1^{|v_j|}]_k} \cdot \left(k^{\ell_j} - k^{|u_{j+1}|} k^{\ell_{j+1}}\right)
        \end{split}
    \end{equation*}

    are both Presburger-definable functions of $(k^{\ell_j})_j$. Therefore so is:

    \begin{equation*}
        \begin{split}
            [\pi_a(w)]_k & = [\pi_a(u_1) 0^{\ell_1}]_k + [\pi_a(v_1)^{e_1} 0^{|u_2| + \ell_2}]_k \\
            & + [\pi_a(u_2) 0^{\ell_2}]_k + [\pi_a(v_2)^{e_2} 0^{|u_3| + \ell_3}]_k \\
            & + \dots \\
            & + [\pi_a(u_r) 0^{\ell_r}]_k + [\pi_a(v_r)^{e_r} 0^{|u_{r+1}|}]_k \\
            & + [\pi_a(u_{r+1})]_k
        \end{split}
    \end{equation*}

    and therefore so is:

    $$[w]_k = ([\pi_1(w)]_k, \dots, [\pi_d(w)]_k)$$

    We thus define $f_i$ to map $(k^{\ell_j})_j$ to $[w]_k$ and note that $f_i(D_i \cap (k^\N)^r) = [L_i]_k$.

    It remains to construct $D$ and $f$ yielding all of $X$. Note that the value of $r$ may depend on $i$, so henceforth we write $r_i$ for the arity of $D_i, f_i$. Assume $X$ is nonempty (otherwise, take $D = \varnothing$), and let $p \in X$ be chosen arbitrarily. Let $r = 1 + \max_i r_i$, and define:

    $$D = (D_1 \times \N^{r - r_1 - 1} \times \{k^1\}) \cup \dots \cup (D_m \times \N^{r - r_m - 1} \times \{k^m\})$$
    $$f(x_1, \dots, x_r) = \begin{cases}
        f_1(x_1, \dots, x_{r_1}) & \text{if } x_r = k^1 \\
        f_2(x_1, \dots, x_{r_2}) & \text{if } x_r = k^2 \\
        \vdots & \vdots \\
        f_m(x_1, \dots, x_{r_m}) & \text{if } x_r = k^m \\
        p & \text{otherwise.}
    \end{cases}$$

    We observe that $D$ and $f$ are Presburger-definable and that $f(D \cap (k^\N)^r) = X$.
\end{proof}

Note that as promised in Section 2, we obtain the following corollary.

\begin{cor}
\label{sparse_sparse_definable}
    Let $X \subseteq \N^d$ be a sparse $k$-automatic set. Then $X$ is sparse-definable.
\end{cor}
\begin{proof}
    By (II) of Theorem \ref{sparse_equals_power_internal}, there exist $r > 0$, a Presburger-definable $D \subseteq \N^r$, and a Presburger-definable $f : D \to \N^d$ such that $f(D \cap (k^\N)^r) = X$. Observe that the image $f(D \cap (k^\N)^r)$ is definable in $(\N, +, k^\N)$.
\end{proof}

\section{Sparse-definable injections}

We begin this section by defining a partial preorder on the sparse-definable sets.

\begin{defn}
    Let $X, Y$ be sparse-definable sets of any arity. We write $X \preccurlyeq Y$ to indicate that there is a sparse-definable injection $f : X \hookrightarrow Y$. We define $\succcurlyeq, \prec, \succ$ accordingly.
\end{defn}

Note that this is indeed a preorder, because the identity function is sparse-definable, and the composition of sparse-definable injections is a sparse-definable injection.

We establish some basic results on this preorder.

\begin{lem}
\label{prec_powers}
    Let $d \in \N$. Then:
    
    \begin{enumerate}[(i)]
        \item $(k^\N)^d \prec (k^\N)^{d+1}$;
        \item $(k^\N)^d \prec \N$.
    \end{enumerate}
\end{lem}
\begin{proof}
    \,
    \begin{enumerate}[(i)]
        \item The sparse-definable injection from $(k^\N)^d$ into $(k^\N)^{d+1}$ is trivial. Assume toward contradiction the existence of a sparse-definable $f : (k^\N)^{d+1} \hookrightarrow (k^\N)^d$. By Fact \ref{kautomatic_linear_growth}, there is a constant $C$ such that each $(x_1, \dots, x_{d+1}) \in (k^\N)^{d+1}$ with $\max (x_1, \dots, x_{d+1}) \leq n$ gives rise via $f$ to a unique $(y_1, \dots, y_d)$ in $(k^\N)^d$ with $\max (y_1, \dots, y_d) \leq Cn$. Assume without loss of generality that $C = k^c$ for some natural number $c$; then for sufficiently large $i$, the number of elements of $(k^\N)^d$ with coordinates at most $Ck^i = k^{c + i}$, which is $(c+i+1)^d$, is at least equal to the number of elements of $(k^\N)^{d+1}$ with coordinates at most $k^i$, which is $(i+1)^{d+1}$. But $(c+i+1)^d$ cannot possibly exceed $(i+1)^{d+1}$ as $i$ increases.
        \item The case $d = 0$ is trivial. Otherwise, let $T$ be the set of types of $d$-tuples from a totally ordered set, and let $S$ be the set of all sets of at most $d$ powers of $k$. Then a given tuple $\vec{x}$ can be uniquely identified by its corresponding $(t(\vec{x}), s(\vec{x})) \in T \times S$, where $t(\vec{x})$ is the type of the order relations between its coordinates and $s(\vec{x})$ is the set of unique values it contains. The set $T$ is finite; enumerate it via a function $\iota : T \to \{0, \dots, |T| - 1\}$.

        Define $f : (k^\N)^d \to \N$ by letting:
        
        $$f(\vec{x}) = |T| \sum_{k^i \in s(\vec{x})} k^i + \iota(t(\vec{x}))$$
        
        The function $f$ is sparse-definable because it may be defined piecewise for every possible value of $t$. Moreover, from the sum of a set of powers of $k$, we can recover each individual power by examining the base-$k$ expansion of the resulting number, so $f$ is an injection.

        It remains to note that $\N \not\preccurlyeq (k^\N)^d$, because the restriction of such an injection to $(k^\N)^{d+1}$ would violate (i).
    \end{enumerate}
\end{proof}

The other result in this section is a dichotomy result for unary sets.

\begin{lem}
\label{prec_dichotomy}
    Let $X \subseteq \N$ be sparse-definable. Either $X$ is sparse, or $\N \preccurlyeq X$.
\end{lem}
\begin{proof}
    Note that any Presburger-definable image of $(k^\N)^r$ (in fact any $k$-automatic image thereof) is sparse by Theorem \ref{sparse_equals_power_internal}. Therefore, apply Fact \ref{schulz_definable_dichotomy} with $\calI$ being the class of sparse sets and with $E = k^\N$. If $X$ is not sparse, we obtain that $X$ contains arbitrarily long arithmetic progressions with the same difference. Let $X_{m,c} = \{x \in \N : xm + c \in X\}$. Then some $X_{m,c}$ is thick, i.e. it contains arbitrarily long intervals.

    Define $f : \N \to \N$ as follows: $f(n) = \min \{x \in \N : n \leq x \wedge x, x+1, \dots, x+n \in X_{m,c}\}$. Notice that $f$ is a sparse-definable function and is therefore $k$-automatic. So by Fact \ref{kautomatic_linear_growth}, when $n > 0$ we have $f(n) \leq C n$ for some constant $C$.

    Now for each positive integer $N$, let $e(n)$ for a natural number $n$ be defined such that $k^{e(n) N}$ is the smallest power of $k^N$ exceeding $n$. Define $g_N(n) = f(k^{e(n) N}) + n$. By definition of $f$, we have $g_N(n) \in X_{m,c}$. Moreover, $g_N$ is sparse-definable; the set $\{k^{eN} : e \in \N\}$ can be defined as the set of powers of $k$ congruent to $1$ modulo $k^N - 1$. Note that:

    $$k^{e(n) N} \leq f(k^{e(n) N}) \leq g_N(n) \leq Ck^{e(n) N} + n < Ck^{e(n) N} + k^{e(n) N} = (C+1)k^{e(n) N}$$

    We choose $N$ such that $(C+1) \leq k^N$; then $k^{e(n) N} \leq g_N(n) < k^{(e(n)+1)N}$.

    We claim $g_N$ is strictly increasing (and therefore injective). Let $n_1 < n_2$. If $e(n_1) = e(n_2)$, then:

    \begin{equation*}
        \begin{split}
            g_N(n_1) & = f(k^{e(n_1)N}) + n_1 \\
            & = f(k^{e(n_2)N}) + n_1 \\
            & < f(k^{e(n_2)N}) + n_2 \\
            & = g_N(n_2).
        \end{split}
    \end{equation*}

    If $e(n_1) < e(n_2)$, then:

    \begin{equation*}
        \begin{split}
            g_N(n_1) & < k^{(e(n_1)+1)N} \\
            & \leq k^{e(n_2)N} \\
            & \leq g_N(n_2).
        \end{split}
    \end{equation*}

    The case $e(n_1) > e(n_2)$ is impossible by definition of $e(n)$.

    Lastly, we define $h : \N \to X$ as $h(n) = g_N(n) \cdot m + c$. This is a sparse-definable injection from $\N$ to $X$.
\end{proof}

As a corollary, we obtain a new proof of (i) in Corollary \ref{entropy_dichotomy}.

\begin{proof}
    Let $X \subseteq \N$ be sparse-definable but not sparse. By Lemma \ref{prec_dichotomy}, there is a sparse-definable injection $f : \N \hookrightarrow X$. Apply Fact \ref{kautomatic_linear_growth} to obtain some $C$ such that $f(n) \leq C \max(1, n)$ for all $n$. Therefore, for nonzero $n$, $|X_{\leq Cn}| \geq |\N_{\leq n}| = n + 1$. Hence $X$ has a growth function in $\Theta(n)$.
\end{proof}

\section{Sparse-definability vs. definability in $\calE$}

We will prove Theorem \ref{sublinear_e_definable} in this section. Fix some expansion $\calE$ of $(\N, +, k^\N)$ such that every unary set definable in $\calE$ is sparse-definable. We establish several results about the definable sets in $\calE$.

\begin{lem}
\label{sparse_e_definable}
    Let $X \subseteq \N^d$ be a subset of a sparse $k$-automatic set $Y$. If $X$ is definable in $\calE$, then $X$ is sparse-definable.
\end{lem}
\begin{proof}
    By Theorem \ref{sparse_equals_power_internal}, there exist $r > 0$ and Presburger-definable $D \subseteq \N^r$ and $f : D \to \N^d$ such that $f(D \cap (k^\N)^r) = Y$. By Lemma \ref{prec_powers}, there exists a sparse-definable injection $g : (k^\N)^r \hookrightarrow \N$. Define:

    $$X' = \{n \in \N : n \in g((k^\N)^r) \wedge g^{-1}(n) \in D \wedge f(g^{-1}(n)) \in X\}$$

    Clearly $X'$ is definable in $\calE$. Because $X'$ is unary, it is sparse-definable. Then:

    $$X = \{\vec{x} \in \N^d : \exists n \in X' \: (f(g^{-1}(n)) = \vec{x})\}$$

    demonstrates that $X$ is sparse-definable.
\end{proof}

\begin{lem}
\label{power_function_e_definable}
    Let $r > 0$, and let $D' \subseteq (k^\N)^r$ and $f : D' \to \N$ be definable in $\calE$. Then $D'$ and $f$ are sparse-definable.
\end{lem}
\begin{proof}
    The set $D'$ is sparse-definable by Lemma \ref{sparse_e_definable}. Moreover, the image $f(D')$ is a unary $\calE$-definable set, so it is sparse-definable. If $f(D')$ is sparse, then the graph of $f$ is a subset of the sparse $k$-automatic set $(k^\N)^r \times f(D')$ and is therefore sparse-definable by Lemma \ref{sparse_e_definable}. Otherwise, by Lemma \ref{prec_dichotomy}, we have some injection $g : \N \hookrightarrow f(D')$.
    
    For $n \in f(D')$, define $h(n)$ to be the lexicographically least $\vec{x} \in D'$ with $f(\vec{x}) = n$. Let $j : (k^\N)^{r+1} \hookrightarrow \N$ be the sparse-definable injection given by Lemma \ref{prec_powers}. Then $F = h \circ g \circ j$ is an injection from $(k^\N)^{r+1}$ into $D' \subseteq (k^\N)^r$.

    Now, $F$ is definable in $\calE$ by inspection, and its graph is a subset of the sparse $k$-automatic set $(k^\N)^{2r+1}$. By Lemma \ref{sparse_e_definable}, $F$ is sparse-definable. But this is now a sparse-definable injection $(k^\N)^{r+1} \hookrightarrow (k^\N)^r$, contradicting Lemma \ref{prec_powers}.
\end{proof}

\begin{lem}
\label{edefinable_linear_growth}
    Let $f : \N^r \to \N^d$ be a function definable in $\calE$. Then there exists some $C$ such that whenever $f(x_1, \dots, x_r) = (y_1, \dots, y_d)$, we have $\max(y_1, \dots, y_d) \leq C \max(1, x_1, \dots, x_r)$.
\end{lem}
\begin{proof}
    Define $g : k^\N \to \N$ as follows: $g(k^n)$ is the maximum coordinate of any $f(\vec{x})$ with $\vec{x}$ having maximum coordinate at most $k^n$. This function is definable in $\calE$, so by the previous lemma, it is sparse-definable. Extend $g$ to a function $g^* : \N \to \N$ by mapping all non-powers to zero; this is still sparse-definable. By Fact \ref{kautomatic_linear_growth}, we have a linear bound $g^*(k^n) = g(k^n) \leq C' k^n$.

    Let $f(x_1, \dots, x_r) = (y_1, \dots, y_d)$, and let $k^n$ be the smallest power of $k$ that is at least equal to $\max(x_1, \dots, x_r)$. Then:

    \begin{equation*}
        \begin{split}
             \max(y_1, \dots, y_d) & \leq g(k^n) \\
             & \leq C' k^n \\
             & \leq C' k \max(1, x_1, \dots, x_r).
        \end{split}
    \end{equation*}

    Therefore, we need only let $C = C' k$.
\end{proof}

We now prove the main theorem.

\begin{thm}
\label{sublinear_e_definable}
    Let $X \subseteq \N^d$ be definable in $\calE$, and assume that $|X_{\leq n}| \in o(n)$. Then $X$ is sparse-definable.
\end{thm}
\begin{proof}
    Let $X \subseteq \N^d$ be definable in $\calE$, and assume that $|X_{\leq n}| \in o(n)$.

    Choose some $1 \leq i \leq d$, and consider the projection $\pi_i(X)$. Define $f : \N \to \N^d$ as follows: if $n \in \pi_i(X)$, then $f(n)$ is the lexicographically least element of $X$ with $n$ as its $i$th coordinate; otherwise we set $f(n) = \vec{0}$. By Lemma \ref{edefinable_linear_growth}, the coordinates of $f(n)$ do not exceed $C \max (1, n)$ for some $C$. Assume $n \geq 1$; then for any $y \in \pi_i(X)$ with $y \leq n$, there exists a $\vec{x} \in X$ with $\pi_i(\vec{x}) = y$ with coordinates at most $Cn$. Therefore, for nonzero $n$, $|X_{\leq Cn}| \geq |(\pi_i(X))_{\leq n}|$; in particular, the latter expression is also in $o(n)$.

    Note that $\pi_i(X)$ is unary and definable in $\calE$, so it is sparse-definable. By Corollary \ref{entropy_dichotomy}, $\pi_i(X)$ is either sparse or has $\Theta(n)$ growth function, but its growth function is $o(n)$. So $\pi_i(X)$ must be sparse for each $i$. Therefore $X$ is a subset of the sparse $k$-automatic set $\pi_1(X) \times \dots \times \pi_d(X)$. We have $X$ sparse-definable by Lemma \ref{sparse_e_definable}.
\end{proof}

\bibliographystyle{plain}
\bibliography{biblio}

\end{document}